\documentclass[11pt]{article}
\usepackage{graphicx} 
\usepackage{amssymb, amsmath, amsfonts,latexsym, oldgerm,amscd,amsthm}
\usepackage [english]{babel}
\usepackage[utf8]{inputenc}
\usepackage [autostyle, english = american]{csquotes}
\MakeOuterQuote{"}
\usepackage{relsize}
\usepackage{mathtools}
\usepackage{url}
\usepackage[makeroom]{cancel}
\usepackage[margin=2.8cm]{geometry}
\usepackage{comment}

\newtheorem{theorem}{Theorem}[section]
\newtheorem{lemma}[theorem]{Lemma}

\newtheorem{definition}[theorem]{Definition}

\newtheorem{remark}[theorem]{Remark}

\newcommand{\ETrans}{\mathrm{ETrans}}
\newcommand{\Trans}{\mathrm{Trans}}
\newcommand{\En}{\mathrm{E}}
\newcommand{\ESp}{\mathrm{ESp}}
\newcommand{\ETransSp}{\mathrm{ETrans_{Sp}}}   
\newcommand{\TransSp}{\mathrm{Trans_{Sp}}}   
\newcommand{\EO}{\mathrm{EO}}
\newcommand{\ETransO}{\mathrm{ETrans_{O}}}   
\newcommand{\TransO}{\mathrm{Trans_{O}}}

\title{Normality of Relative Elementary Tranvection Groups}
\author{Sunil Rampuria, Ruddarraju Amrutha, Pratyusha Chattopadhyay}

\begin{document}

\maketitle

\begin{abstract}
\noindent A. A. Suslin and V. I. Kopeiko proved that the elementary linear, symplectic, and orthogonal groups are normal in the general linear group, symplectic group, and orthogonal group respectively. They also proved relative versions of these normality results with respect to an ideal of a ring. A. Bak, R. Basu, and R. A. Rao proved that the linear, symplectic, and orthogonal transvection groups are normal in the respective automorphism groups. These results are generalizations of the normality results by Suslin and Kopeiko in the setup of modules. In this paper, we prove stronger versions of the normality results proved by Bak, Basu, and Rao, which say that the relative transvection groups are normal in the respective automorphism groups.
\end{abstract}



\section{Introduction}

\noindent In 1955, J-P. Serre in \cite{Serre} posed a fundamental question: is every finitely generated projective module over a polynomial ring over a field necessarily free? Known as Serre's problem on projective modules, this conjecture profoundly shaped the development of algebraic $K$-theory. It was affirmatively proved independently by A. A. Suslin in \cite{SusVas} and D. Quillen in \cite{Quil} in 1976. Their work initiated a systematic exploration of the structure of linear groups over rings. In 1977, using his solution to Serre's conjecture, Suslin (in \cite{Sus}) established the $K_1$-analogue of Serre's conjecture, proving that for any field $k$ and \( n \geq 3 \), every element in the 
special linear group \( \operatorname{SL}_n(k[x_1, \dots ,x_r]) \) can be written as a product of elementary linear matrices,
thereby extending classical results about matrices over fields to polynomial rings. 

\vskip 3mm
\noindent A cornerstone of this theory is the study of the normality theorems which assert that the elementary subgroup is normal within the full general linear, symplectic, and orthogonal group. For the general linear group over a commutative ring $R$ and an ideal $I$ of $R$, the normality of $\operatorname{E}_n(R,I)$ in $\operatorname{GL}_n(R)$, for $n \geq 3$ was proved by Suslin in \cite{Sus} in 1977 . Analogous results were subsequently developed by V. I. Kopeiko for the elementary symplectic group in \cite{Kop} in 1978, and for the elementary orthogonal group by Kopeiko and Suslin in \cite{KopSus} in 1982.

\vskip 3mm
\noindent In 1964, H. Bass introduced two types of linear transvections of a projective module $R$ $\oplus$ $P$ in his foundational work on the cancellation of projective modules (in \cite{Bass}). This construction provided the algebraic machinery needed to formulate and prove some of the "classical" results in $K$-theory (see \cite{Bass3}, \cite{Swan}). Later in 1973, Bass in \cite{Bass3} extended this approach, introducing two analogous types of symplectic transvections of a symplectic module $\mathbb{H}(R) \oplus (P, \langle , \rangle)$, and two types of orthogonal transvections of an orthogonal module $\mathbb{H}(R) \oplus (P, \langle , \rangle)$, establishing a parallel framework for modules equipped with a form (these are recalled in Section~3, Section~4 and Section~5). These generalize the classical elementary generators, and open the door to studying linear, symplectic, and orthogonal groups over projective modules.

\vskip 3mm
\noindent Despite their distinct origins, the structural theory of Bass’s transvection groups is connected to the classical normality results of Suslin and Kopeiko. A natural question arises: can the classical normality theorems be extended to these more general transvection groups? A significant step in this direction was taken in 2010 by A. Bak, R. Basu, and R.A. Rao, who in \cite{BakBasuRao} established a local-global principle for the transvection groups in the absolute case to prove the normality of the transvection groups in the absolute case. 

\vskip 3mm
\noindent In this paper, we generalize the classical normality theorems to Bass's relative transvection groups. We prove that for any commutative ring $R$, an ideal $I$ of $R$, and any finitely generated projective $R$-module $M$, the group $\operatorname{ETrans}(M,IM)$ is normal in $\operatorname{Aut}(M)$. Furthermore, we establish analogous normality results for the elementary symplectic transvection group within the full symplectic group of $\mathbb{H}(R) \oplus (P, \langle , \rangle)$, and for the elementary orthogonal transvection group within the full orthogonal group. We give two different proofs of these results; one through direct computation, and the other by leveraging the relative version of the local-global principle stated by Bak, Basu, and Rao. These results serve as a natural extension of the theorems of Suslin and Kopeiko from the setting of free modules to the broader context of projective modules equipped with a form, thereby unifying and advancing the structural study of linear, symplectic, and orthogonal groups over commutative rings.

\vskip 3mm
\noindent Although direct computation suffices to prove the normality for the relative case, our primary motivation for including the local-global proof is methodological.
It replicates the successful strategy from the absolute case, thereby establishing a consistent framework for proving normality across our entire family of groups. This alternative proof highlights the structural role of the local-global principle, which will be essential for future developments in this area of mathematics.

\section{Preliminaries}

\noindent Throughout, $R$ will represent a commutative ring with unity, and $I$ will denote an ideal of $R$. We let $R^n$ denote the space of column vectors of length $n$ with entries in $R$. The ring of matrices of size $r \times s$ with entries in $R$ is denoted by $M_{rs}(R)$. The ring of matrices of size $n \times n$ with entries in $R$ is denoted by $M_n(R)$. For $\alpha \in M_{rs}(R)$, we will denote by $\alpha^{\top},$ the transpose of $\alpha$, which is a matrix in $M_{sr}(R)$. The identity matrix of size $n \times n$ is denoted by $I_n$, and $e_{ij}$ denotes the $n \times n$ matrix which has 1 in the $(i, j)$-th position and $0$ everywhere else. For $\alpha \in M_m(R)$, and $\beta \in M_n(R)$, the matrix $\begin{pmatrix}
\alpha & 0 \\
0 & \beta
\end{pmatrix}$, which is an element of $M_{m+n}(R)$, is denoted by $\alpha \perp \beta$. If $\mathfrak{P}$ is a prime ideal of $R$, the localization of $R$ at $\mathfrak{P}$ is denoted by $R_\mathfrak{P}$. For an $R$-module $M$, $1\in\mathrm{Aut}(M)$ denotes the identity map from $M$ to $M$.

\vskip 2mm
\begin{definition} 
\rm{For $n\geq 2$, the \textit{elementary linear group} E$_{n}(R)$ is the subgroup of SL$_{n}(R)$ generated by the elementary matrices $E_{ij}(a) = I_{n} + ae_{ij}$, where $i\neq j$ and $a\in R$. 

\vskip 2mm 
\noindent
The \textit{true relative elementary subgroup} E$_{n}(I)$ is defined to be the subgroup of E$_{n}(R)$ generated by the elementary matrices $E_{ij}(x)$ with $x$ $\in$ $I$ (See \cite{Nica}).

\vskip 2mm 
\noindent
The \textit{relative elementary subgroup} E$_{n}(R,I)$ is the normal closure of E$_{n}(I)$ in E$_{n}(R)$. This is generated by elements of the form $E_{kl}(a)E_{ij}(x)E_{kl}(-a)$, where $a \in R$, and $x \in I$ (see \cite{Vas}, Lemma 8).}

\end{definition}
\vskip 3mm
\begin{definition} 
\rm{Given an invertible alternating matrix $\varphi$ of order $2n$, \[\mathrm{Sp}_{\varphi}(R)\coloneqq\left\{M\in\text{GL}_{2n}(R)\mid M^{\top}\varphi M=\varphi \right\}.\]

\vskip 2mm
\noindent The matrix $\psi_{n}$ = $\sum\limits_{i=1}^{n} e_{2i-1,2i} - \sum\limits_{i=1}^{n} e_{2i,2i-1}$ is called the \textit{standard alternating matrix} of order $2n$. 
\vskip 2mm
\noindent Also note that $\psi_{n}=\begin{pmatrix}
\psi_{n-1} & 0 \\
0 & \psi_{1}
\end{pmatrix}$, where $\psi_{1}=\begin{pmatrix}
0 & 1 \\
-1 & 0
\end{pmatrix}$.

\vskip 3mm
\noindent \textit{Symplectic group} of order $2n$ is defined as Sp$_{2n}(R)\coloneqq\left\{M\in\text{GL}_{2n}(R)\mid M^{\top}\psi_{n}M=\psi_{n}\right\}$. An element of the symplectic group is called a \textit{symplectic matrix}.}
\end{definition}

\vskip 2mm
\begin{definition}  
\rm{ Let $\sigma$ be the permutation of $\{1,2,\dots,2n\}$ defined as $\sigma(2i) = 2i-1$, $\sigma(2i-1) = 2i$. For $z \in R$, $1\leq i\neq j \leq 2n$, we set up (as in \cite{Kop})
	\begin{center}
	    $se_{ij}(z)$=
    $\begin{cases}
      I_{2n} + ze_{ij} & \text{if } i = \sigma(j),\\
      I_{2n} + ze_{ij} -(-1)^{i+j}ze_{\sigma(j)\sigma(i)} & \text{if } i \neq \sigma(j).\\
    \end{cases}$
    \end{center}
\vskip 2mm 

\noindent Note that $se_{ij}(z)$ belong to Sp$_{2n}(R)$, and are called \textit{elementary symplectic matrices}. Subgroup of Sp$_{2n}(R)$ generated by $se_{ij}(z)$ is called \textit{elementary symplectic group}, and is denoted by ESp$_{2n}(R)$.

\vskip 2mm
\noindent The \textit{true relative elementary symplectic subgroup} ESp$_{2n}(I)$ is defined to be the subgroup of ESp$_{2n}(R)$ generated by the elementary symplectic matrices $se_{ij}(x)$, where $x\in I$.

\vskip 2mm
\noindent The \textit{relative elementary symplectic subgroup} ESp$_{2n}(R,I)$ is the normal closure of ESp$_{2n}(I)$ in ESp$_{2n}(R)$. This is generated by elements of the form $se_{kl}(a)se_{ij}(x)se_{kl}(-a)$, where $a \in R$, and $x \in I$.
}
\end{definition}

\vskip 3mm

\begin{definition} \rm{Given an invertible symmetric matrix $\widetilde{\varphi}$ of order $2n$, \[\mathrm{O}_{\widetilde{\varphi}}(R)\coloneqq\left\{M\in\text{GL}_{2n}(R)\mid M^{\top}\widetilde{\varphi} M=\widetilde{\varphi} \right\}.\]

\vskip 2mm
\noindent The matrix $\widetilde{\psi}_{n}$ = $\sum\limits_{i=1}^{n} e_{2i-1,2i} + \sum\limits_{i=1}^{n} e_{2i,2i-1}$ is called the \textit{standard symmetric matrix} of order $2n$.

\vskip 2mm
\noindent Also note that $\widetilde{\psi}_{n}=\begin{pmatrix}
\widetilde{\psi}_{n-1} & 0 \\
0 & \widetilde{\psi}_{1}
\end{pmatrix}$, where $\widetilde{\psi}_{1}=\begin{pmatrix}
0 & 1 \\
1 & 0
\end{pmatrix}$.

\vskip 2mm
\noindent \textit{Orthogonal group} of order $2n$ is defined as O$_{2n}(R)=\left\{M\in\text{GL}_{2n}(R)\mid M^{\top}\widetilde{\psi}_{n}M=\widetilde{\psi}_{n}\right\}$. An element of the orthogonal group is called an \textit{orthogonal matrix}.}
\end{definition}

\vskip 3mm
\begin{definition} \rm{For $z \in R$, $1\leq i\neq j \leq 2n, i\neq \sigma(j)$, set up (as in \cite{KopSus})
\[oe_{ij}(z) = I_{2n} + ze_{ij} - ze_{\sigma(j)\sigma(i)}.\] 

\vskip 3mm
\noindent Note that $oe_{ij}(z)$ belong to O$_{2n}(R)$, and are called \textit{elementary orthogonal matrices}. Subgroup of O$_{2n}(R)$ generated by $oe_{ij}(z)$ is called elementary orthogonal group, and is denoted by EO$_{2n}(R)$.

\vskip 2mm
\noindent The \textit{true relative elementary orthogonal subgroup} EO$_{2n}(I)$ is defined to be the subgroup of EO$_{2n}(R)$ generated by the elementary orthogonal matrices $oe_{ij}(x)$, where $x \in$ $I$.

\vskip 2mm
\noindent The \textit{relative elementary orthogonal subgroup} EO$_{2n}(R,I)$ is the normal closure of EO$_{2n}(I)$ in EO$_{2n}(R)$. This is generated by elements of the form $oe_{kl}(a)oe_{ij}(x)oe_{kl}(-a)$, where $a \in R$, and $x \in I$.}
\end{definition}

\section{Linear Transvections}

In \cite{BakBasuRao}, Bak, Basu, and Rao proved that ETrans$(M)$ is a normal subgroup of Aut$(M)$. In this section, we extend this result to a relative case with respect to an ideal $I$ of $R$ (see Theorem \ref{RLT normality}). We will show that the relative elementary transvection group ETrans$(M,IM)$ is normal in Aut$(M)$. We give two proofs of this result, one through direct computation of conjugates, and the other using a local-global principle for the relative elementary transvection group.

\begin{definition} \rm{ An element $v=(v_{1},\dots,v_{n})^{\top}\in R^{n}$ is said to be \textit{unimodular} if there are elements $w_{1},\dots,w_{n}\in R$ such that $\sum\limits_{i=1}^{n}v_{i}w_{i}=1$. We denote by Um$_{n}(R)$, the set of all unimodular elements $v\in R^{n}$. Set of all unimodular elements in $R^{n}$ which are congruent to $e_{1}=(1,0,\dots,0)^{\top}$ modulo $I$ is denoted by Um$_{n}(R,I)$ (if $I=R$, then Um$_{n}(R,I)$ is Um$_{n}(R)$).}
\end{definition}

\vskip 2mm
\begin{definition} \rm{ Given any finitely generated (in particular, finitely generated projective) $R$-module $P$, an element $p\in P$ is said to be \textit{unimodular} if there exists an $R$-linear map $\theta:P\to R$ such that $\theta(p) = 1$. Equivalently, $p$ is unimodular if $Rp$ is a free submodule of rank $1$, and is a direct summand of $P$. The collection of all unimodular elements of $P$ is denoted by Um$(P)$.

\vskip 2mm
\noindent Let $P$ be of the form $R\oplus Q$, and have the element of the form $(1,0)$ which corresponds to the unimodular element. An element $(a,q)\in P$ is said to be relatively unimodular with respect to $I$ if $(a,q)$ is unimodular and congruent to $(1,0)$ modulo $IP$. The collection of all relative unimodular elements with respect to $I$ is denoted by Um$(P,IP)$. }
\end{definition}

\vskip 2mm
\noindent Following Bass ($\cite{Bass}$), we recall the definition of linear transvection group here.

\vskip 2mm
\begin{definition} \rm{ Let $M$ be a finitely generated $R$-module. Let $q\in M$, and $\pi \in M^{*}=$ Hom$(M,R)$, with $\pi(q)=0$. Let $\pi_{q}:M\to M$, be defined as $\pi_{q}(p)=\pi(p)q$ for any $p\in M$. An automorphism of the form $1+\pi_{q}$ is called a \textit{linear transvection} of $M$ if either $q\in$ Um$(M)$ or $\pi \in$ Um$(M^{*})$. Subgroup of Aut$(M)$ generated by these transvections is denoted by Trans$(M)$. 

\vskip 2mm
\noindent The group of \textit{relative transvections} with respect to an ideal $I$ is generated by the transvections of the form $1+\pi_{q}$, where either $q\in IM$ or $\pi \in IM^{*}$. The group generated by relative transvections is denoted by Trans$(M, IM)$.
}\end{definition}

\vskip 2mm
\begin{definition} \rm{ Let $M$ be a finitely generated $R$-module. The automorphism of $N=(R\oplus M)$ of the form \[(a,p)^{\top}\mapsto (a,p+ax)^{\top},\] or of the form \[(a,p)^{\top}\mapsto (a+\tau(p),p)^{\top},\] where $x\in M$, and $\tau\in M^{*}$ are called \textit{elementary transvections} of $N$. We denote the first automorphism by $E_{x}$, and the second one by $E_{\tau}^{*}$. It can be checked that these are transvections of $N$. Consider $\pi(t,y)=t,\ q=(0,x)$ to get $E_{x}$, and consider $\pi(a,p)=\tau(p),\ q=(1,0)$ to get $E_{\tau}^{*}$. The subgroup of Trans$(N)$ generated by elementary transvections is denoted by ETrans$(N)$.

\vskip 2mm
\noindent The elementary transvections of $N=(R\oplus M)$ of the form $E_{x},E_{\tau}^{*}$ where $x\in IM$, and $\tau \in IM^{*}$ are called \textit{relative elementary transvections} with respect to an ideal $I$, and the group generated by them is denoted by ETrans$(IN)$. The normal closure of ETrans$(IN)$ in ETrans$(N)$ is denoted by ETrans$(N,IN)$. }\end{definition}

\vskip 2mm
\noindent The following result proved by Suslin is a particular case of Theorem \ref{RLT normality}, in the case of a free module.

\begin{lemma}[\cite{Sus}, Corollary 1.4]
\label{En normality}
    $\mathrm{E}_{n}(R,I)$ is normal in $\mathrm{GL}_{n}(R)$, for $n \geq 3$.
\end{lemma}

\noindent The following result shows that the linear transvection groups are generalizations of the elementary linear group in the setup of modules.

\begin{lemma}[\cite{ChatRao}, Lemma 4.5]
\label{relt=rlt=En}   
Let $I$ be an ideal of $R$. Let $M$ be a free $R$-module of rank $n \geq 2$, and let $N = (R \oplus M)$. Then
\[
\ETrans(N, IN) = \Trans(N, IN) = \En_{n+1}(R, I).
\]
\end{lemma}

\begin{lemma}[\cite{ChatRao}, Lemma 4.9]
   Let $M$ be a finitely generated $R$-module, and $\alpha \in \Trans(M)$. Then there exists $\beta(X) \in \Trans(M[X])$ such that $\beta(1) = \alpha$, and $\beta(0) = Id$.

\end{lemma}

 \noindent We  establish a relative version of Lemma 3.7.

\begin{lemma}
\label{lin homotopy}
   Let $M$ be a finitely generated $R$-module, and $\alpha \in \Trans(M,IM)$. Then there exists $\beta(X) \in \Trans(M[X],IM[X])$ such that $\beta(1) = \alpha$, and $\beta(0) = Id$.

\end{lemma}

\noindent \textit{Proof :} Note that $\alpha \in \Trans(M, IM)$ is a product of elements of the form $1 + \pi_q$, where $\pi \in M^{*}$, $q \in M$, $\pi(q) = 0$ with either $q\in$ Um$(M)$ or $\pi \in$ Um$(M^{*})$, and either $\pi \in IM^{*}$ (if $q\in$ Um$(M)$) or $q \in IM$ (if $\pi \in$ Um$(M^{*})$).
Let $\pi X$ denote $\pi$ times $X$, which belongs to $M^{*}[X]$, and let $q X$ denote $q$ times $X$, which belongs to $M[X]$. Define $\pi X(q) = \pi(q)X$ or $\pi(qX) = \pi(q)X$.

\vskip 2mm
\noindent We set $\beta(X)$ to be the product of elements of the form $1 + \pi_{(qX)}$ or $1+ (\pi X)_q$, whenever $Id + \pi_q$ appears in the expression of $\alpha$. The choice is made as follows: if $\pi$ is unimodular (hence $q \in IM$), we replace $q$ by $qX$ yielding $Id + \pi_{(qX)}$. If $q$ is unimodular (hence $\pi \in IM^*$), we replace $\pi$ by $\pi X$ yielding $Id + (\pi X)_q$. Each lifted generator lies in $\operatorname{Trans}(M[X], IM[X])$ because the relative condition is preserved: in the first case $qX \in IM[X]$; in the second case $\pi X \in IM^*[X]$. Hence, we get $\beta(X) \in \operatorname{Trans}(M[X], IM[X])$ such that $\beta(1) = \alpha$, and $\beta(0) = Id$.
$\hfill \square$

\begin{lemma}[\cite{ChatRao}, Proposition 4.10]
\label{rlt=relt}
  Let $I$ be an ideal of $R$. Let $M$ be a finitely generated projective $R$-module of rank at least $2$, and $N = (R \oplus M)$. Then $\operatorname{Trans}(N, IN) = \operatorname{ETrans}(N, IN)$.

\end{lemma}

\noindent In the following lemma, we state a local-global principle for the relative elementary transvection group.
\begin{lemma}[\cite{ChatRao}, Lemma 4.7]
\label{lin LG}   
Let $I$ be an ideal of $R$, and let $M$ be a finitely generated projective $R$-module of rank $n \geq 2$. Let $N = (R \oplus M)$. Let $\alpha(X) \in \operatorname{Aut}(N[X])$, with $\alpha(0) = Id$. If for each maximal ideal $\mathfrak{m}$ of $R$, $\alpha(X)_{\mathfrak{m}} \in \operatorname{ETrans}(N_{\mathfrak{m}}[X], IN_{\mathfrak{m}}[X])$, then $\alpha(X) \in \operatorname{ETrans}(N[X], IN[X])$.

\end{lemma}

\noindent Now we prove the main result of this section, first by direct calculations. Then, we give a second proof of the result using local-global principle.

\begin{theorem}
\label{RLT normality}
     Let $I$ be an ideal of $R$. Let $M$ be a finitely generated projective $R$-module of rank at least $2$, and $N = (R \oplus M)$. Then $\operatorname{ETrans}(N, IN)$ is normal in $\operatorname{Aut}(N)$.
\end{theorem}

\noindent \textit{Proof 1:} 
Note that $\operatorname{ETrans}(N,IN)=\operatorname{Trans}(N,IN)$ by Lemma 3.9. So, it suffices to prove that $\operatorname{Trans}(N,IN)$ is normal in $\operatorname{Aut}(N)$. 

\vskip 3mm
\noindent Let $\sigma\in \mathrm{Trans}(N,IN)$ be a relative transvection. Then $\sigma=1+\pi_{q}$ for some $q\in N$, $\pi \in N^{*}$, $\pi(q)=0$, and either $q\in$ Um$(N)$, $\pi \in IN^{*}$ or $q\in IN$, $\pi \in$ Um$(N^{*})$. For any $p\in N$, the map $\pi_{q}:N\to N$ is defined as $\pi_{q}(p)=\pi(p)q$.

\vskip 3mm
\noindent Let $\gamma\in\operatorname{Aut}(N)$. We will show that $\gamma\sigma\gamma^{-1} \in \operatorname{Trans}(N,IN)$.
\vskip 3mm
\noindent Let $p\in N$
\begin{eqnarray*}
	\gamma\sigma\gamma^{-1}(p)
	&=&\gamma(1+\pi_{q})\gamma^{-1}(p)\\
	&=&\gamma(\gamma^{-1}(p)+\pi_{q}(\gamma^{-1}(p)))\\
    &=&\gamma(\gamma^{-1}(p)+\pi(\gamma^{-1}(p))q)\\
    &=&p+\gamma(\pi(\gamma^{-1}(p))q)\\
	&=&p+\pi(\gamma^{-1}(p))\gamma(q)\hfill \text{ [as $\pi(\gamma^{-1}(p))\in R$, and $\gamma$ is $R$-linear]}\\
    &=&p+\mu(p)\gamma(q)\hfill \text{ [$\mu=(\pi\circ\gamma^{-1})$, and $\mu\in N^*$]}\\
	&=&p+\mu_{\gamma(q)}(p)\\
    &=&(Id_{N} +\mu_{\gamma(q)})(p)\\
    \implies \gamma\sigma\gamma^{-1}&=&(1 +\mu_{\gamma(q)})
\end{eqnarray*}

\noindent Note that $q\in IN$, $\pi \in$ Um$(N^{*})\Rightarrow\gamma(q)\in IN,$ $\mu \in$ Um$(N^{*}), $, and $q\in$ Um$(N)$, $\pi\in IN^* \Rightarrow \gamma(q)\in$ Um$(N)$, $\mu\in IN^*$. Also, as $\pi(q)=0$, we have 
\[\mu(\gamma(q))=\pi\circ\gamma^{-1}(\gamma(q))=\pi(q)=0.\]
Therefore, $\gamma\sigma\gamma^{-1}$ is a relative linear transvection, and hence is in $\mathrm{Trans}(N,IN)$. As $\operatorname{Trans}(N,IN)$ is generated by the relative transvections, it follows that it is a normal subgroup of $\operatorname{Aut}(N)$.  $\hfill \square$

\vskip 4mm
\noindent \textit{Proof 2:} 
Let $\delta \in \operatorname{ETrans}(N, IN)$. Note that $\operatorname{ETrans}(N, IN)=\operatorname{Trans}(N, IN)$ by Lemma \ref{rlt=relt}, and hence $\delta \in \operatorname{Trans}(N, IN)$. By Lemma \ref{lin homotopy}, there exists $\delta(X) \in \operatorname{Trans}(N[X], IN[X])$ such that $\delta(1)=\delta$, and $\delta(0)=Id$. For a $\gamma \in$ Aut$(N)$, we define $\beta(X)=\gamma\delta(X)\gamma^{-1}$. 

\vskip 2mm
\noindent Note that for every maximal ideal $\mathfrak{m}$ of $R$, over the local ring $R_\mathfrak{m}$, the projective module $N_\mathfrak{m}$ is free of $\text{rank} \geq 3$. Therefore, over $R_\mathfrak{m}$ we have Aut$(N_\mathfrak{m})=$GL$_{n+1}(R_\mathfrak{m})$, and by Lemma \ref{relt=rlt=En}, $\operatorname{ETrans}(N_{\mathfrak{m}}[X], IN_{\mathfrak{m}}[X])=\operatorname{E}_{n+1}(R_{\mathfrak{m}}[X],I_{\mathfrak{m}}[X])$.

\vskip 2mm
\noindent Note that by Lemma \ref{En normality}, for all maximal ideals $\mathfrak{m}$ of $R$, we have \[\beta(X)_{\mathfrak{m}}=\gamma_{\mathfrak{m}}\delta(X)_{\mathfrak{m}}\gamma^{-1}_{\mathfrak{m}}\in\operatorname{E}_{n+1}(R_{\mathfrak{m}}[X],I_{\mathfrak{m}}[X])=\operatorname{ETrans}(N_{\mathfrak{m}}[X], IN_{\mathfrak{m}}[X]).\] Hence, by Lemma \ref{lin LG}, we have $\beta(X)\in \operatorname{ETrans}(N[X], IN[X])$, and hence $\beta(1)=\gamma\delta\gamma^{-1}\in \operatorname{ETrans}(N, IN)$. Therefore, $\operatorname{ETrans}(N, IN)$ is a normal subgroup of Aut$(N)$. $\hfill \square$

\section{Symplectic Transvections}
In \cite{BakBasuRao}, Bak, Basu, and Rao proved that $\operatorname{ETrans}_{\operatorname{Sp}}(Q, \langle,\rangle)$ is a normal subgroup of Sp$(Q)$. In this section, we extend this result to a relative case with respect to an ideal $I$ of $R$ (see Theorem \ref{RST normality}). We show that the relative symplectic transvection group $\operatorname{ETrans}_{\operatorname{Sp}}(Q, IQ, \langle , \rangle)$ is normal in Sp$(Q)$. We give two proofs of this result, one through direct computation, and one using a local-global principle for the relative elementary transvection group.

\begin{definition} \rm{ A \textit{symplectic $R$-module} is a pair $(P, \langle , \rangle)$, where $P$ is a finitely generated projective $R$-module of even rank, and $\langle , \rangle : P \times P \to R$ is a non-degenerate (i.e., $P \cong P^*$ by $x \to \langle x, -\rangle$) alternating bilinear form. }\end{definition}

\vskip 2mm
\begin{definition} \rm{An \textit{isometry} of a symplectic module $(P, \langle , \rangle)$ is an automorphism of $P$ which fixes the bilinear form. The group of isometries of $(P, \langle , \rangle)$ is denoted by $\operatorname{Sp}(P)$. }\end{definition}

\vskip 2mm
\begin{definition} \rm{ A \textit{symplectic transvection} of a symplectic module $(P, \langle , \rangle)$ (as defined by Bass in \cite{Bass3}) is an automorphism of the form
\[
\sigma(p) = p + \langle u, p \rangle v + \langle v, p \rangle u + \alpha \langle u, p \rangle u,
\]
where $\alpha \in R$, $u, v \in P$ are fixed elements with $\langle u, v \rangle = 0$, and either $u$ or $v$ is unimodular. It is easy to check that $\langle \sigma(p), \sigma(q) \rangle = \langle p, q \rangle$, and $\sigma$ has an inverse $\tau(p) = p - \langle u, p \rangle v - \langle v, p \rangle u - \alpha \langle u, p \rangle u$.

\vskip 2mm
\noindent The subgroup of $\operatorname{Sp}(P)$ generated by the symplectic transvections is called \textit{symplectic transvection group}, and is denoted by  $\operatorname{Trans_{Sp}}(P, \langle , \rangle)$ (see \cite{Swan}, Page 35).

\vskip 2mm
\noindent The group of \textit{relative symplectic transvections} with respect to an ideal $I$ is generated by the symplectic transvections of the form 
\[
\sigma(p) = p + \langle u, p \rangle v + \langle v, p \rangle u + \alpha \langle u, p \rangle u,
\]
where $\alpha \in I$, $u \in P$, $v \in IP$ are fixed elements with $\langle u, v \rangle = 0$. The group generated by relative symplectic transvections is denoted by $\operatorname{Trans_{Sp}}(P, IP, \langle , \rangle)$.
 }\end{definition}
\vskip 2mm 

\noindent The following remark gives a symplectic $R[X]$-module structure on $P[X]$, for a symplectic $R$-module $(P,\langle,\rangle)$. For details, see section 4 of \cite{AC3}.
\begin{remark}
\label{P[X] is symplectic}
    If $(P,\langle,\rangle)$ is a symplectic $R$-module, then $P[X]$ is a finitely generated $R[X]$-module of even rank. The map $\langle,\rangle_{\otimes R[X]}: P[X]\times P[X]\rightarrow R[X]$ given by
    \[\bigg\langle\sum_{i}p_iX^i,\sum_{j}q_jX^j\bigg\rangle_{\otimes R[X]}=\sum_{i,j}\langle p_i,q_j\rangle X^{i+j}\]
    is a non-degenerate alternating bilinear form. The pair $(P[X],\langle,\rangle_{\otimes R[X]})$ is a symplectic $R[X]$- module.
\end{remark}

\vskip 2mm
\begin{definition} \rm{ Let $(P_1, \langle, \rangle_1)$, and $(P_2, \langle, \rangle_2)$ be two symplectic $R$-modules. Their \textit{orthogonal sum} is a pair $(P, \langle, \rangle)$, where $P = P_1 \oplus P_2$, and the bilinear form is defined by \[
\langle(p_1, p_2), (q_1, q_2)\rangle = \langle p_1, q_1\rangle_1 + \langle p_2, q_2\rangle_2.\] }\end{definition}

\begin{remark}
There is a non-degenerate alternating bilinear form $\langle , \rangle$ on the $R$-module $\mathbb{H}(R) = R \oplus R^*$, namely \[\langle (a_1, f_1), (a_2, f_2) \rangle = f_2(a_1) - f_1(a_2).\]
\end{remark}

\noindent \textbf{Notation.} Now onwards \( Q \) will denote \( (R^2 \oplus P) \) with induced form on \( (\mathbb{H}(R) \oplus P) \), and \( Q[X] \) will denote \( (R[X]^2 \oplus P[X]) \) with induced form on \( (\mathbb{H}(R[X]) \oplus P[X]) \).

\vskip 2mm
\begin{definition} \rm{ The automorphisms of $Q$ of the form
\[
(a, b, p)^{\top} \mapsto (a, b - \langle p, q \rangle - \alpha a, p - aq)^{\top},
\]
or of the form
\[
(a, b, p)^{\top} \mapsto (a + \langle p, q \rangle + \beta b, b, p - bq)^{\top},
\]
where $\alpha, \beta \in R$, and $q \in P$, are called \textit{elementary symplectic transvections}. We denote the first isometry by $\rho(q, \alpha)$, and the second one by $\mu(q, \beta)$. It can be checked that the elementary symplectic transvections are symplectic transvections on $Q$. Consider $(u, v) = ((0, 1, 0), (0, 0, q))$ to get $\rho(q, \alpha)$, and consider $(u, v) = ((-1, 0, 0), (0, 0, q))$ to get $\mu(q, \beta)$.

\vskip 2mm
\noindent The subgroup of $\operatorname{Trans_{Sp}}(Q, \langle , \rangle)$ generated by the elementary symplectic transvections is denoted by $\operatorname{ETrans_{Sp}}(Q, \langle , \rangle)$. }\end{definition}

\vskip 2mm
\noindent The elementary symplectic transvections of $Q$ of the form $\rho(q, \alpha)$, $\mu(q, \beta)$, where $q \in IP$, and $\alpha, \beta \in I$ are called \textit{relative elementary symplectic transvections} with respect to an ideal $I$.

\vskip 2mm
\noindent The subgroup of $\operatorname{ETrans}_{\operatorname{Sp}}(Q, \langle , \rangle)$ generated by relative elementary symplectic transvections is denoted by $\operatorname{ETrans}_{\operatorname{Sp}}(IQ, \langle , \rangle)$. The normal closure of $\operatorname{ETrans}_{\operatorname{Sp}}(IQ, \langle , \rangle)$ in $\operatorname{ETrans}_{\operatorname{Sp}}(Q, \langle , \rangle)$ is denoted by $\operatorname{ETrans}_{\operatorname{Sp}}(Q, IQ, \langle , \rangle)$.

\vskip 2mm
\noindent The following result, proved by Kopeiko, is a particular case of Theorem \ref{RST normality}, in the case of a free module.

\begin{lemma}[\cite{Kop}, Corollary 1.11]
\label{Kop normality}
    $\mathrm{ESp}_{2n}(R,I)$ is normal in $\mathrm{Sp}_{2n}(R)$, for $n \geq 2$.
\end{lemma}

\noindent The following result gives a relation between an alternating matrix of Pfaffian $1$, and the standard alternating matrix, over a local ring.

\begin{lemma}[\cite{ChatRao}, Lemma 5.2]
\label{phi conjugate of psi}
Let \((R, \mathfrak{m})\) be a local ring, and \(I\) be an ideal of \(R\). Let \(\varphi\) be an alternating matrix of Pfaffian \(1\) over \(R\), such that \(\varphi \equiv \psi_n \pmod{I}\). Then \(\varphi\) is of the form \((1 \perp \varepsilon)^\top \psi_n (1 \perp \varepsilon)\), for some \(\varepsilon \in \mathrm{E}_{2n-1}(R, I)\).
\end{lemma}

\begin{remark}[\cite{ChatRao}, Remark 5.4]
\label{phi over local}
Let \(\varphi\) be an alternating matrix of Pfaffian \(1\), over \(R\). Let \(\mathfrak{m}\) be a maximal ideal of \(R\). We will get \(\varepsilon(\mathfrak{m}) \in \mathrm{E}_{2n-1}(R_{\mathfrak{m}})\) such that over \(R_{\mathfrak{m}}\) we have \(\varphi = (1 \perp \varepsilon(\mathfrak{m}))^\top \psi_n (1 \perp \varepsilon(\mathfrak{m}))\). Let \(a\) be the product of denominators of all the entries of \(\varepsilon(\mathfrak{m})\). Clearly \(a\) is not in \(\mathfrak{m}\). Hence, we get \(\varepsilon\) from \(\mathrm{E}_{2n-1}(R_a)\) such that \(\varphi = (1 \perp \varepsilon)^\top \psi_n (1 \perp \varepsilon)\). When dealing with the relative case with respect to an ideal \(I\) of \(R\), we will always assume that the alternating matrix \(\varphi\) of Pfaffian \(1\) is congruent to \(\psi_n \pmod{I}\). Using Lemma \ref{phi conjugate of psi}, we get that over ring \(R_{\mathfrak{m}}\) we have \(\varphi = (1 \perp \varepsilon)^\top \psi_n (1 \perp \varepsilon)\), where \(\varepsilon \in \mathrm{E}_{2n-1}(R_a, I_a)\), for some \(a \notin \mathfrak{m}\).
\end{remark}

\begin{remark}[\cite{ChatRao}, Remark 5.12]
Let \(P = \oplus_{i=1}^{2n} Re_i\) be a free \(R\)-module. The non-degenerate alternating bilinear form \(\langle , \rangle\) on \(P\) corresponds to an alternating matrix \(\varphi\) with respect to the basis \(\{e_1, e_2, \ldots, e_{2n}\}\) of \(P\), and we write \(\langle p, q \rangle = p^\top \varphi q\). In this case the symplectic transvection \(\sigma(p) = p + \langle u, p \rangle v + \langle v, p \rangle u + \alpha \langle u, p \rangle u\) corresponds to the matrix \((I_{2n} + vu^\top \varphi + uv^\top \varphi)(I_{2n} + \alpha uu^\top \varphi )\), and the group generated by them is denoted by \(\operatorname{Trans_{Sp}}(P, \langle , \rangle_\varphi)\). Also in this case \(\operatorname{ETrans_{Sp}}(Q, \langle , \rangle_{\psi_1 \perp \varphi})\) will be generated by the matrices of the form\\ \(\rho_\varphi(q, \alpha) = \begin{pmatrix} 1 & 0 & 0 \\ -\alpha & 1 & q^\top \varphi \\ -q & 0 & I_{2n} \end{pmatrix}\), and 
\(\mu_\varphi(q, \beta) = \begin{pmatrix} 1 & \beta & -q^\top \varphi\\ 0 & 1 & 0 \\ 0 & -q & I_{2n} \end{pmatrix}\).
\end{remark}

\noindent The following result shows that the relative transvection groups are generalizations of the relative elementary symplectic group.

\begin{lemma}[\cite{ChatRao}, Lemma 5.13, 5.14]
\label{trans sp=esp}
    Let $R$ be a commutative ring with $R = 2R$, and let $I$ be an ideal of $R$. 
Let $P$ be a free $R$-module of rank $2n$, $n \geq 1$. Keeping the notation of the previous remark, if $\varphi = \psi_n$, the standard alternating matrix, then
\[
\ETransSp(Q, IQ, \langle,\rangle_{\psi_{n+1}}) = \TransSp(Q, IQ, \langle,\rangle_{\psi_{n+1}}) = \ESp_{2n+2}(R, I).
\]
\end{lemma}

\noindent The following result gives a relation between the symplectic groups $\mathrm{Sp}_\varphi(R)$, and $\mathrm{Sp}_{\varphi^*}(R)$, when $\varphi$ and $\varphi^*$ are conjugates of each other.

\begin{lemma}[\cite{ChatRao2}, Lemma 3.6]
\label{sym groups are conjugates}
    Let $\varphi$ and $\varphi^*$ be two invertible alternating matrices such that 
$\varphi = (1 \perp \varepsilon)^\top \varphi^* (1 \perp \varepsilon)$, for some 
$\varepsilon \in \operatorname{E}_{2n-1}(R)$. Then we have
\[
\operatorname{Sp}_{\varphi}(R) = (1 \perp \varepsilon)^{-1} \operatorname{Sp}_{\varphi^*} (R) (1 \perp \varepsilon).
\]
\end{lemma}

\begin{lemma}[\cite{ChatRao}, Lemma 5.16]
\label{sym trans groups are conjugates}
Let $I$ be an ideal of $R$, and $P$ be a free $R$-module of rank $2n$. Let $(P, \langle \ , \ \rangle_{\varphi})$, and $(P, \langle \ , \ \rangle_{\varphi^*})$ be two symplectic $R$-modules with $\varphi = (1 \perp \varepsilon)^\top \varphi^* (1 \perp \varepsilon)$, for some $\varepsilon \in \mathrm{E}_{2n-1}(R, I)$. Then
\[
\operatorname{Trans_{Sp}}(P, IP, \langle \ , \ \rangle_{\varphi}) = (1 \perp \varepsilon)^{-1} \operatorname{Trans_{Sp}}(P, IP, \langle \ , \ \rangle_{\varphi^*}) (1 \perp \varepsilon),
\]
\[
\operatorname{ETrans_{Sp}}(Q, IQ, \langle \ , \ \rangle_{\psi_{1} \perp \varphi}) = (I_3 \perp \varepsilon)^{-1} \operatorname{ETrans_{Sp}}(Q, IQ, \langle \ , \ \rangle_{\psi_{1} \perp \varphi^*}) (I_3 \perp \varepsilon).
\]
\end{lemma}

\vskip 0.1mm
\noindent \begin{lemma}[\cite{ChatRao}, Lemma 5.22]
    Let $(P, \langle, \rangle)$ be a symplectic $R$-module, and $\alpha \in \TransSp(P, \langle, \rangle)$. 
Then there exists $\beta(X) \in \TransSp(P[X], \langle, \rangle_{\otimes R[X]})$ such that 
$\beta(1) = \alpha$, and $\beta(0) = Id$.

\end{lemma}

\noindent We establish a relative version of Lemma 4.15.

\begin{lemma}
\label{sym homotopy}
    Let $(P, \langle, \rangle)$ be a symplectic $R$-module, and $\alpha \in \TransSp(P, IP, \langle, \rangle)$. 
Then there exists $\beta(X) \in \TransSp(P[X], IP[X], \langle, \rangle_{\otimes R[X]})$ such that 
$\beta(1) = \alpha$, and $\beta(0) = Id$.

\end{lemma}

\noindent \textit{Proof:} Note that $\alpha \in \TransSp(P, IP, \langle , \rangle)$ is a product of relative symplectic transvections $\sigma$, given by \[\sigma(p)=p + \langle u, p \rangle v + \langle v, p \rangle u + a \langle u, p \rangle u,\] where $a \in I$, $u \in P$, $v \in IP$ are fixed elements with $\langle u, v \rangle = 0$, and $u$ is unimodular. Define the map $\sigma X$ as follows: 
\[\text{For }p\in P,\text{ set }\sigma X(p)=p + \langle u, p \rangle vX + \langle vX, p \rangle u + aX \langle u, p \rangle u.\] 
\[\text{Define }\sigma X\bigg(\sum_ip_iX^i\bigg)=\sum_i\sigma X(p_i)X^i.\]
Here, $vX$ represents $v$ times $X$, and $aX$ represents $a$ times $X$. Note that $aX\in I[X]$, $u\in P[X]$, $vX\in IP[X]$ are elements with 
$\langle u,vX\rangle_{\otimes R[X]}=\langle u,v\rangle X=0$,
and $u$ is unimodular. Therefore, $\sigma X$ is a relative symplectic transvection. We set $\beta(X)$ to be the product of elements of the form $\sigma X$, whenever $\sigma$ appears in the expression of $\alpha$. Then $\beta(1) = \alpha$, and $\beta(0) = Id$. Moreover, each such generator lies in $\TransSp(P[X], IP[X], \langle , \rangle)$, since $vX \in IP[X]$, and $aX \in I[X]$. Hence, $\beta(X) \in \TransSp(P[X], IP[X], \langle , \rangle)$. $\hfill \square$

\begin{lemma}[\cite{ChatRao}, Theorem 5.23]
\label{rst=rest}    Let $R$ be a commutative ring with $R = 2R$, and let $I$ be an ideal of $R$. Let $(P, \langle , \rangle)$ be a symplectic $R$-module with $P$ a finitely generated projective module of rank $2n$, $n \geq 1$. Also assume that over the local ring $R_{\mathfrak{m}}$ for any maximal ideal $\mathfrak{m}$ of $R$, the alternating form $\langle , \rangle$ corresponds to the alternating matrix $\varphi_{\mathfrak{m}}$ (with respect to some basis), where $\varphi_{\mathfrak{m}} \equiv \psi_n \pmod{I}$. Then $\TransSp(Q, IQ, \langle , \rangle) = \ETransSp(Q, IQ, \langle , \rangle)$.
\end{lemma}

\noindent In the following lemma, we state a local-global principle for the relative elementary transvection group, for the symplectic case.

\begin{lemma}[\cite{ChatRao}, Lemma 5.20]
\label{sym LG}
    Let $R$ be a commutative ring with $R = 2R$, and let $I$ be an ideal of $R$. Let $(P, \langle , \rangle)$ be a symplectic $R$-module with $P$ being a finitely generated projective module of rank $2n$, $n \geq 1$. Let $\alpha(X) \in \operatorname{Sp}(Q[X])$, with $\alpha(0) = Id$. Also assume that over the local ring $R_{\mathfrak{m}}$ for any maximal ideal $\mathfrak{m}$ of $R$, the alternating form $\langle , \rangle$ corresponds to the alternating matrix $\varphi_{\mathfrak{m}}$ of Pfaffian $1$ (with respect to some basis), where $\varphi_{\mathfrak{m}} \equiv \psi_n \pmod{I}$. If for each maximal ideal $\mathfrak{m}$ of $R$, $\alpha(X)_{\mathfrak{m}} \in \ETransSp(Q_{\mathfrak{m}}[X], IQ_{\mathfrak{m}}[X], \langle , \rangle_{\psi_1 \perp \varphi_{\mathfrak{m}}})$, then $\alpha(X) \in \ETransSp(Q[X], IQ[X], \langle , \rangle)$.
\end{lemma}

\noindent Now we establish the main result of this section, first by calculations, and then by using the local-global principle (Lemma \ref{sym LG}).
\begin{theorem}
\label{RST normality}
    Let $R$ be a commutative ring with $R = 2R$, and let $I$ be an ideal of $R$. Let $(P, \langle , \rangle)$ be a symplectic $R$-module with $P$ being a finitely generated projective module of rank $2n$, $n \geq 1$.  Also assume that for any maximal ideal $\mathfrak{m}$ of $R$, over the local ring $R_{\mathfrak{m}}$, the alternating form $\langle , \rangle$ corresponds to the alternating matrix $\varphi_{\mathfrak{m}}$ of Pfaffian $1$ (with respect to some basis), where $\varphi_{\mathfrak{m}} \equiv \psi_n \pmod{I}$. Then $\operatorname{ETrans}_{\operatorname{Sp}}(Q, IQ, \langle , \rangle)$ is normal in $\operatorname{Sp}(Q)$.
\end{theorem}

\noindent \textit{Proof 1:} Note that $\operatorname{ETrans}_{\operatorname{Sp}}(Q, IQ, \langle , \rangle)=\operatorname{Trans}_{\operatorname{Sp}}(Q, IQ, \langle , \rangle)$ by Lemma \ref{rst=rest}. So, it suffices to prove that $\operatorname{Trans}_{\operatorname{Sp}}(Q, IQ, \langle , \rangle)$ is normal in $\operatorname{Sp}(Q)$.
Let $\gamma\in\operatorname{Sp}(Q)$, and $\sigma\in\rm{Trans}_{\rm{Sp}}$$(Q,IQ,\langle,\rangle)$ be a relative symplectic transvection. By definition, we have
\begin{equation*}
	\sigma(p)=p+\langle u,p\rangle v+\langle v,p \rangle u+\alpha\langle u,p \rangle u,
\end{equation*}
where $\alpha\in I$, $u\in Q$, and $v\in IQ$ with $\langle u, v\rangle=0$, and either $u$ or $v$ is unimodular. 

\vskip 2mm
\noindent Let $p\in Q$, and $q=\gamma^{-1}(p)$. Then,
\begin{eqnarray*}
	\gamma\sigma\gamma^{-1}(p)
	&=&\gamma\sigma(q)\\
	&=&\gamma (q+\langle u,q\rangle v+\langle v,q\rangle u+\alpha\langle u,q \rangle u)\\
	&=&\gamma(q)+\langle u,q\rangle \gamma(v)+\langle v,q\rangle \gamma(u)+\alpha\langle u,q \rangle \gamma(u)\hfill \text{ [as $\gamma$ is $R$-linear]}\\
	&=&p+\langle u,\gamma^{-1}(p)\rangle \gamma(v)+\langle v,\gamma^{-1}(p)\rangle \gamma(u)+\alpha\langle u,\gamma^{-1}(p) \rangle \gamma(u).
\end{eqnarray*}
Let $u_1=\gamma(u)$, and $v_1=\gamma(v)$. Then,
\begin{eqnarray*}
    \gamma\sigma\gamma^{-1}(p)
    &=&p+\langle \gamma^{-1}(u_1),\gamma^{-1}(p)\rangle v_1+\langle \gamma^{-1}(v_1),\gamma^{-1}(p)\rangle u_1+\alpha\langle \gamma^{-1}(u_1),\gamma^{-1}(p) \rangle u_1\\
    &=&p+\langle u_1,p\rangle v_1+\langle v_1,p\rangle u_1+\alpha\langle u_1,p \rangle u_1\text{ [as $\gamma^{-1}$ is an isometry]}
\end{eqnarray*}
Note that $\alpha\in I$, $u_1\in Q$, and $v_1\in IQ$ . As $\gamma$ is an isometry, $\langle u_1,v_1\rangle=\langle \gamma(u),\gamma(v)\rangle=\langle u,v\rangle=0$. Either $u_1$ or $v_1$ is unimodular. Therefore, $\gamma\sigma\gamma^{-1}$ is a relative symplectic transvection, and hence $\operatorname{Trans}_{\operatorname{Sp}}(Q,IQ,\langle,\rangle)$ is a normal subgroup of $\operatorname{Sp}(Q)$.
 $\hfill \square$

\vskip 4mm
\noindent \textit{Proof 2:}
Let $\delta \in \operatorname{ETrans}_{\operatorname{Sp}}(Q, IQ, \langle , \rangle)$. Note that $\operatorname{ETrans}_{\operatorname{Sp}}(Q, IQ, \langle , \rangle)=\operatorname{Trans}_{\operatorname{Sp}}(Q, IQ, \langle , \rangle)$ by Lemma \ref{rst=rest}, and hence $\delta \in \operatorname{Trans}_{\operatorname{Sp}}(Q, IQ, \langle , \rangle)$. 

\vskip 2mm
\noindent By Lemma \ref{sym homotopy}, there exists $\delta(X) \in \operatorname{Trans}_{\operatorname{Sp}}(Q[X], IQ[X], \langle , \rangle_{\otimes R[X]})$ such that $\delta(1)=\delta$, and $\delta(0)=Id$. For $\gamma \in \operatorname{Sp}(Q)$, let $\gamma(X)(\sum q_jX^j)=\sum \gamma(q_j)X^j$. Note that $\gamma(X)\in\mathrm{Sp}(Q[X])$. Define $\beta(X)=\gamma(X)\delta(X)\gamma(X)^{-1}$. 

\vskip 2mm
\noindent Note that for every maximal ideal $\mathfrak{m}$ of $R$, the projective $R_\mathfrak{m}$-module $Q_\mathfrak{m}$ is free of $\text{rank} \geq 4$, over the local ring $R_\mathfrak{m}$. Since $\varphi_{\mathfrak{m}} \equiv \psi_n \pmod{I}$, by Lemma \ref{phi conjugate of psi}, and Remark \ref{phi over local}, over  the ring \(R_{\mathfrak{m}}\) we have \(\varphi_\mathfrak{m} = (1 \perp \varepsilon)^\top \psi_n (1 \perp \varepsilon)\), where \(\varepsilon \in \mathrm{E}_{2n-1}(R_a, I_a)\), for some \(a \notin \mathfrak{m}\). So, by Lemma \ref{sym trans groups are conjugates}, we get $\operatorname{Trans}_{\operatorname{Sp}}(Q_{\mathfrak{m}}[X], IQ_{\mathfrak{m}}[X], \langle , \rangle_{\psi_{1}\perp\varphi_\mathfrak{m}})=(I_{3} \perp \varepsilon)^{-1}\operatorname{Trans}_{\operatorname{Sp}}(Q_{\mathfrak{m}}[X], IQ_{\mathfrak{m}}[X], \langle , \rangle_{\psi_{n+1}})(I_{3} \perp \varepsilon)$, and by Lemma \ref{sym groups are conjugates}, we have $\operatorname{Sp}_{\psi_{1}\perp\varphi_\mathfrak{m}}(R_\mathfrak{m}[X])=(I_{3} \perp \varepsilon)^{-1}\operatorname{Sp}_{\psi_{n+1}}(R_\mathfrak{m}[X])(I_{3} \perp \varepsilon)$.

\vskip 2mm
\noindent Note that by Lemma \ref{Kop normality}, we have $\operatorname{ESp}_{2n+2}(R_{\mathfrak{m}}[X],I_{\mathfrak{m}}[X])$ is normal in $\operatorname{Sp}_{2n+2}(R_\mathfrak{m}[X])$, and by Lemma \ref{trans sp=esp}, we have $\operatorname{Trans}_{\operatorname{Sp}}(Q_{\mathfrak{m}}[X], IQ_{\mathfrak{m}}[X], \langle , \rangle_{\psi_{n+1}})=\operatorname{ESp}_{2n+2}(R_{\mathfrak{m}}[X],I_{\mathfrak{m}}[X])$. We show that for any $y'\in\operatorname{Trans}_{\operatorname{Sp}}(Q_{\mathfrak{m}}[X], IQ_{\mathfrak{m}}[X], \langle , \rangle_{\psi_{1}\perp\varphi_\mathfrak{m}})$, and $z'\in\operatorname{Sp}_{\psi_{1}\perp\varphi_\mathfrak{m}}(R_\mathfrak{m}[X])=\operatorname{Sp}(Q_\mathfrak{m}[X])$, the conjugate $z' y' (z')^{-1}$ will lie in $\operatorname{Trans}_{\operatorname{Sp}}(Q_{\mathfrak{m}}[X], IQ_{\mathfrak{m}}[X], \langle , \rangle_{\psi_{1}\perp\varphi_\mathfrak{m}})$. Let $\theta = (I_{3} \perp \varepsilon)$. Then $y'=\theta^{-1} y \theta$, and $z'=\theta^{-1} z \theta$ for some $y\in\operatorname{ESp}_{2n+2}(R_{\mathfrak{m}}[X],I_{\mathfrak{m}}[X])$, and $z\in\operatorname{Sp}_{\psi_{n+1}}(R_\mathfrak{m}[X])=\operatorname{Sp}_{2n+2}(R_\mathfrak{m}[X])$. Then
\[
\begin{aligned}
z' y' (z')^{-1} &= (\theta^{-1} z \theta)(\theta^{-1} y \theta)(\theta^{-1} z^{-1} \theta) \\
&= \theta^{-1} z y z^{-1} \theta\\
&= \theta^{-1} w \theta \text{\hphantom{fas}  for some } w\in \operatorname{Trans}_{\operatorname{Sp}}(Q_{\mathfrak{m}}[X], IQ_{\mathfrak{m}}[X], \langle , \rangle_{\psi_{n+1}}) \text{ by Lemma \ref{Kop normality}} \\
&= w' \text{\hphantom{fasslllll}  for some } w'\in \operatorname{Trans}_{\operatorname{Sp}}(Q_{\mathfrak{m}}[X], IQ_{\mathfrak{m}}[X], \langle , \rangle_{\psi_{1}\perp\varphi_\mathfrak{m}}).
\end{aligned}
\]

\noindent So, for all maximal ideals $\mathfrak{m}$ of $R$, we get that the term $\beta(X)_{\mathfrak{m}}=\gamma(X)_{\mathfrak{m}}\delta(X)_{\mathfrak{m}}\gamma^{-1}(X)_{\mathfrak{m}}$ belongs to $\operatorname{Trans}_{\operatorname{Sp}}(Q_{\mathfrak{m}}[X], IQ_{\mathfrak{m}}[X], \langle , \rangle_{\psi_{1}\perp\varphi_\mathfrak{m}})=\operatorname{ETrans}_{\operatorname{Sp}}(Q_{\mathfrak{m}}[X], IQ_{\mathfrak{m}}[X], \langle , \rangle_{\psi_{1}\perp\varphi_\mathfrak{m}})$. By Lemma \ref{sym LG}, $\beta(X)$ belongs to  $\ETransSp(Q[X], IQ[X], \langle , \rangle)$, and  $\beta(1)=\gamma\delta\gamma^{-1}$ belongs to $\operatorname{ETrans}_{\operatorname{Sp}}(Q, IQ, \langle , \rangle)$. Therefore, $\operatorname{ETrans}_{\operatorname{Sp}}(Q, IQ, \langle , \rangle)$ is a normal subgroup of $\operatorname{Sp}(Q)$.$\hfill \square$

\section{Orthogonal case}

In \cite{BakBasuRao}, Bak, Basu, and Rao proved that $\operatorname{ETrans}_{\operatorname{O}}(Q, \langle , \rangle)$ is a normal subgroup of O$(Q)$. In this section, we extend this result to a relative case with respect to an ideal $I$ of $R$ (see Theorem \ref{ROT normality}). We will show that the relative orthogonal transvection group $\operatorname{ETrans}_{\operatorname{O}}(Q, IQ, \langle , \rangle)$ is normal in O$(Q)$. We will be giving two proofs of this result, one through direct computation, and one by using local-global principle.

\begin{definition} \rm{ An \textit{orthogonal $R$-module} is a pair $(P, \langle , \rangle)$, where $P$ is a finitely generated projective $R$-module of even rank, and $\langle , \rangle : P \times P \to R$ is a non-degenerate (i.e., $P \cong P^*$ by $x \to \langle x, -\rangle$) symmetric bilinear form. This is also known as \textit{symmetric inner product space}.}\end{definition}

\vskip 2mm
\begin{definition} \rm{An \textit{isometry} of an orthogonal module $(P, \langle , \rangle)$ is an automorphism of $P$ which fixes the bilinear form. The group of isometries of $(P, \langle , \rangle)$ is denoted by $\operatorname{O}(P)$. }\end{definition}

\vskip 2mm
\begin{definition} \rm{ An \textit{orthogonal transvection} of an orthogonal module $(P, \langle , \rangle)$ (as defined by Bass in \cite{Bass3}) is an automorphism of the form
\[
\tau(p) = p - \langle u, p \rangle v + \langle v, p \rangle u,
\]
where $u, v \in P$ are isotropic, i.e., $\langle u, u \rangle = \langle v, v \rangle = 0$ with $\langle u, v \rangle = 0$, and either $u$ or $v$ is unimodular. It is easy to check that $\langle \tau(p), \tau(q) \rangle = \langle p, q \rangle$, i.e., $\tau \in \operatorname{O}(P)$, and $\tau$ has an inverse $\sigma(p) = p + \langle u, p \rangle v - \langle v, p \rangle u$.

\vskip 2mm
\noindent The subgroup of $\operatorname{O}(P)$ generated by the orthogonal transvections is called \textit{orthogonal transvection group}, and is denoted by $\operatorname{Trans_{O}}(P, \langle , \rangle)$ (see \cite{Bass3} or \cite{HahnOT}).

\vskip 2mm
\noindent The group of \textit{relative orthogonal transvections} with respect to an ideal $I$ is generated by the orthogonal transvections of the form 
\[
\sigma(p) = p - \langle u, p \rangle v + \langle v, p \rangle u 
\]
where either $u \in IP$ or $v \in IP$. The group generated by relative orthogonal transvections is denoted by $\operatorname{Trans_{O}}(P, IP, \langle , \rangle)$.
 }\end{definition}

\noindent The following remark is an orthogonal analogue of Remark \ref{P[X] is symplectic}, and gives an orthogonal module structure on $P[X]$.
\begin{remark}
    If $(P,\langle,\rangle)$ is an orthogonal $R$-module, then $(P[X],\langle,\rangle_{\otimes R[X]})$ is an orthogonal $R[X]$-module, where $\langle,\rangle_{\otimes R[X]}:P[X]\times P[X]\rightarrow P[X]$ is given by
    \[\bigg\langle\sum_{i}p_iX^i,\sum_{j}q_jX^j\bigg\rangle_{\otimes R[X]}=\sum_{i,j}\langle p_i,q_j\rangle X^{i+j}.\]
\end{remark}

\vskip 2mm
\begin{definition} \rm{ Let $(P_1, \langle, \rangle_1)$ and $(P_2, \langle, \rangle_2)$ be two orthogonal $R$-modules. Their \textit{orthogonal sum} is a pair $(P, \langle, \rangle)$, where $P = P_1 \oplus P_2$, and the inner product is defined by \[
\langle (p_1, p_2), (q_1, q_2) \rangle = \langle p_1, q_1 \rangle_1 + \langle p_2, q_2 \rangle_2.
\]}\end{definition}

\begin{remark}
There is a non-degenerate symmetric bilinear form $\langle , \rangle$ on the $R$-module $\mathbb{H}(R) = R \oplus R^*$, namely \[\langle (a_1, f_1), (a_2, f_2) \rangle = f_2(a_1) + f_1(a_2).\] With this symmetric bilinear form, $\mathbb{H}(R)$ is called \textit{hyperbolic plane}. 
\end{remark}

\noindent \textbf{Notation.} Now onwards \( Q \) will denote \( (R^2 \oplus P) \) with induced form on \( (\mathbb{H}(R) \oplus P) \), and \( Q[X] \) will denote \( (R[X]^2 \oplus P[X]) \) with induced form on \( (\mathbb{H}(R[X]) \oplus P[X]) \).

\vskip 2mm
\begin{definition} \rm{An orthogonal $R$-module $(P, \langle , \rangle)$ is called \textit{split} if there exists a submodule $N \subseteq P$ such that $N$ is a direct summand of $P$, and $N$ is
precisely equal to its orthogonal complement $N^{\perp} = \{p \in P : \langle p, n \rangle = 0 \text{ for all } n \in N\}$.
Moreover, an orthogonal module $(P, \langle , \rangle)$ over the ring $R$ is called \textit{locally split} if
$(P_{\mathfrak{m}}, \langle, \rangle)$ is a split orthogonal $R_{\mathfrak{m}}$-module for every maximal ideal $\mathfrak{m}$ of $R$.
}\end{definition}

\vskip 2mm
\begin{definition} \rm{ The automorphisms of $Q = (R^2 \oplus P)$ of the form
\[
(a, b, p) \mapsto (a, b + \langle p, q \rangle, p - aq),\]
or of the form
\[
(a, b, p) \mapsto (a + \langle p, q \rangle, b, p - bq),
\]
where $a, b \in R$, and $p, q \in P$, are called \textit{elementary orthogonal transvections}. We denote the first isometry by $\rho(q)$, and the second one by $\mu(q)$. It can be checked that the elementary orthogonal transvections are orthogonal transvections on $Q$. Consider $(u, v) = ((0, 1, 0), (0, 0, q))$ to get $\rho(q)$, and consider $(u, v) = ((1, 0, 0), (0, 0, q))$ to get $\mu(q)$.

\vskip 2mm
\noindent The subgroup of $\operatorname{Trans_{O}}(Q, \langle , \rangle)$ generated by elementary orthogonal transvections is denoted by $\operatorname{ETrans_{O}}(Q, \langle , \rangle)$.}\end{definition}

\vskip 2mm
\begin{definition} \rm{ The elementary orthogonal transvections of $Q$ of the form $\rho(q)$, $\mu(q)$, where $q \in IP$ are called \textit{relative elementary orthogonal transvections} with respect to an ideal $I$.

\vskip 2mm
\noindent The subgroup of $\operatorname{ETrans}_{\operatorname{O}}(Q, \langle , \rangle)$ generated by relative elementary orthogonal transvections is denoted by $\operatorname{ETrans}_{\operatorname{O}}(IQ, \langle , \rangle)$. The normal closure of $\operatorname{ETrans}_{\operatorname{O}}(IQ, \langle , \rangle)$ in $\operatorname{ETrans}_{\operatorname{O}}(Q, \langle , \rangle)$ is denoted by $\operatorname{ETrans}_{\operatorname{O}}(Q, IQ, \langle , \rangle)$.}\end{definition}

\vskip 2mm
\noindent The following result, proved by Suslin and Kopeiko, is a particular case of Theorem \ref{ROT normality}, in the case of a free module.

\begin{lemma}[\cite{KopSus}, Corollary 2.13]
\label{KopSus normality}
    $\mathrm{EO}_{2n}(R,I)$ is normal in $\mathrm{O}_{2n}(R)$, for $n \geq 3$.
\end{lemma}

\begin{remark}[\cite{Chat}, Remark 4.10]
Let \(P = \oplus_{i=1}^{2n} Re_i\) be a free \(R\)-module. The non-degenerate symmetric bilinear form \(\langle , \rangle\) on \(P\) corresponds to a symmetric matrix \(\widetilde{\varphi}\) with respect to the basis \(\{e_1, e_2, \ldots, e_{2n}\}\) of \(P\), and we write \(\langle p, q \rangle = p^\top \widetilde{\varphi} q\). In this case the orthogonal transvection \(\tau(p) = p - \langle u, p \rangle v + \langle v, p \rangle u \) corresponds to the matrix \((I_{2n} -  vu^\top \widetilde{\varphi} +  uv^\top \widetilde{\varphi} )\), and the group generated by them is denoted by \(\operatorname{Trans_{O}}(P, \langle , \rangle_{\widetilde{\varphi}})\). Also in this case \(\operatorname{ETrans_{O}}(Q, \langle , \rangle_{\widetilde{\psi_1} \perp \widetilde{\varphi}})\) will be generated by matrices of the form \(\rho_{\widetilde{\varphi}}(q) = \begin{pmatrix} 1 & 0 & 0 \\ 0 & 1 & q^\top\widetilde{\varphi} \\ -q & 0 & I_{2n} \end{pmatrix}\), and 
\(\mu_{\widetilde{\varphi}}(q) = \begin{pmatrix} 1 & 0 & q^\top \widetilde{\varphi} \\ 0 & 1 & 0 \\ 0 & -q & I_{2n} \end{pmatrix}\).
\end{remark}

\noindent The following result shows that the orthogonal transvection groups are a generalization of the elementary orthogonal group.

\begin{lemma}[\cite{Chat},  Lemma 4.11, 4.12]
\label{reot=eo}
    Let $R$ be a commutative ring with $R = 2R$, and let $I$ be an ideal of $R$. 
Let $P$ be a free $R$-module of rank $2n$, $n \geq 1$. Keeping the notation of the previous remark, if $\widetilde{\varphi} = \widetilde{\psi}_{n}$, the standard symmetric matrix, then
\[
\ETransO(Q, IQ, \langle,\rangle_{\widetilde{\psi}_{n+1}}) = \TransO(Q, IQ, \langle,\rangle_{\widetilde{\psi}_{n+1}}) = \EO_{2n+2}(R, I).
\]
\end{lemma}

\noindent We establish a relation between two orthogonal groups, defined with respect to symmetric matrices which are conjugates of each other.

\begin{lemma}
\label{orth conj}
    Let $\widetilde{\varphi}$ and $\widetilde{\varphi}^*$ be two invertible symmetric matrices such that 
$\widetilde{\varphi} = \varepsilon^\top \widetilde{\varphi}^*  \varepsilon$, for some 
$\varepsilon \in \operatorname{GL}_{2n}(R)$. Then we have
\[
\operatorname{O}_{\widetilde{\varphi}}(R) = \varepsilon^{-1} \operatorname{O}_{\widetilde{\varphi}^*} (R) \varepsilon.
\]
\end{lemma}

\noindent \textit{Proof :} Let $\eta \in \operatorname{O}_{\widetilde{\varphi}^*}(R)$. By definition $\eta^\top \widetilde{\varphi}^* \eta = \widetilde{\varphi}^*$. Note that
\[
\begin{aligned}
\left(\varepsilon^{-1} \eta \varepsilon \right)^\top \widetilde{\varphi} \left( \varepsilon^{-1} \eta \varepsilon \right) 
&= \varepsilon^\top \eta^\top (\varepsilon^{-1})^\top \widetilde{\varphi} \varepsilon^{-1} \eta \varepsilon \\
&= \varepsilon^\top \eta^\top \big((\varepsilon^\top)^{-1} \widetilde{\varphi} \varepsilon^{-1}\big) \eta \varepsilon \\
&= \varepsilon^\top \eta^\top \widetilde{\varphi}^* \eta \varepsilon \text{\hspace{0.5cm} [ as }\widetilde{\varphi}^*=(\varepsilon^\top)^{-1}\widetilde{\varphi}\varepsilon^{-1}\text{]}\\
&= \varepsilon^\top \widetilde{\varphi}^* \varepsilon \text{\hspace{0.5cm} [ as }\eta\in\mathrm{O}_{\widetilde{\varphi}^*}(R)\text{]}\\
&= \widetilde{\varphi}.
\end{aligned}
\]
Thus, $\varepsilon^{-1} \eta \varepsilon \in \operatorname{O}_{\widetilde{\varphi}}(R)$ for every $\eta \in \operatorname{O}_{\widetilde{\varphi}^*}(R)$. Hence,
\[
\varepsilon^{-1} \operatorname{O}_{\widetilde{\varphi}^*} (R) \varepsilon \subseteq \operatorname{O}_{\widetilde{\varphi}}(R).
\]

\noindent
For the reverse inclusion, let $\xi \in \operatorname{O}_{\widetilde{\varphi}}(R)$. Then by definition $\xi^\top \widetilde{\varphi} \xi = \widetilde{\varphi}$. From $\widetilde{\varphi} = \varepsilon^\top \widetilde{\varphi}^* \varepsilon$ we obtain $\widetilde{\varphi}^* = (\varepsilon^{-1})^\top \widetilde{\varphi} \varepsilon^{-1}$. Applying the same argument as above with $\widetilde{\varphi}$ and $\widetilde{\varphi}^*$ swapped, and $\varepsilon$ replaced by $\varepsilon^{-1}$, we get $\varepsilon \xi \varepsilon^{-1} \in \operatorname{O}_{\widetilde{\varphi}^*}(R)$, i.e., $\xi \in \varepsilon^{-1} \operatorname{O}_{\widetilde{\varphi}^*} (R) \varepsilon$. Therefore,
\(
\operatorname{O}_{\widetilde{\varphi}}(R) \subseteq \varepsilon^{-1} \operatorname{O}_{\widetilde{\varphi}^*} (R) \varepsilon.
\)
\qed

\begin{lemma}[\cite{Chat}, Lemma 4.14]
\label{reot conj}
Let \(I\) be an ideal of \(R\), and \(P\) be a free \(R\)-module of rank \(2n\). Let
\((P, \langle , \rangle_{\widetilde{\varphi}})\), and \((P, \langle , \rangle_{\widetilde{\varphi}^*})\) be two orthogonal \(R\)-modules with \(\widetilde{\varphi} = \varepsilon^\top \widetilde{\varphi}^* \varepsilon\), for some \(\varepsilon \in \operatorname{GL}_{2n}(R)\). Then
\[
\operatorname{Trans_{O}}(P, IP, \langle , \rangle_{\widetilde{\varphi}}) = \varepsilon^{-1} \operatorname{Trans_{O}}(P, IP, \langle , \rangle_{\widetilde{\varphi}^*}) \varepsilon,
\]  
\[
\operatorname{ETrans_{O}}(Q, IQ, \langle , \rangle_{\widetilde{\psi}_{1} \perp \widetilde{\varphi}}) = (I_2 \perp \varepsilon)^{-1} \operatorname{ETrans_{O}}(Q, IQ, \langle , \rangle_{\widetilde{\psi}_{1} \perp \widetilde{\varphi}^*}) (I_2 \perp \varepsilon).
\]
\end{lemma}

\noindent The following lemma gives an equivalent condition for an orthogonal $R$-module to be split, when every finitely generated projective $R$-module is free.

\begin{lemma}[\cite{Chat}, Lemma 4.18]
\label{split}
Let \(R\) be a ring such that every finitely generated projective module over \(R\) is free. Then an inner product space over \(R\) is split if and only if it possesses a basis so that the associated inner product matrix has the form \(\bigl(\begin{smallmatrix} 0 & I \\ I & A \end{smallmatrix}\bigr)\). If we also assume that 2 is a unit in the ring, then every split inner product space has matrix $\bigl(\begin{smallmatrix} 0 & I \\ I & 0 \end{smallmatrix}\bigr)$ with respect to a suitable basis.
\end{lemma}

\begin{lemma}[\cite{Chat}, Lemma 4.22]
    Let $(P, \langle, \rangle)$ be an orthogonal $R$-module, and $\alpha \in \TransO(P, \langle, \rangle)$. 
Then there exists $\beta(X) \in \TransO(P[X], \langle, \rangle)$ such that 
$\beta(1) = \alpha$, and $\beta(0) = Id$.

\end{lemma}

\noindent We establish a relative version of Lemma 5.16.

\begin{lemma}
\label{orth homotopy}
    Let $(P, \langle, \rangle)$ be an orthogonal $R$-module, and $\alpha \in \TransO(P, IP, \langle, \rangle)$. 
Then there exists $\beta(X) \in \TransO(P[X], IP[X], \langle, \rangle)$ such that 
$\beta(1) = \alpha$, and $\beta(0) = Id$.

\end{lemma}

\noindent \textit{Proof:} Note that $\alpha \in \TransO(P, IP, \langle , \rangle)$ is a product of relative orthogonal transvections of the form $\tau$, where $\tau$ takes $p \in P$ to $p - \langle u, p \rangle v + \langle v, p \rangle u$, where $u,v\in P$ are isotropic, i.e., $\langle u, u \rangle = \langle v, v \rangle = 0$, with $\langle u, v \rangle = 0$, with either $u\in$ Um$(P)$ or $v\in$ Um$(P)$, and either $v \in IP$ (if $u\in$ Um$(P)$) or $u \in IP$ (if  $v\in$ Um$(P)$). Define $\tau X$ as follows: \[\text{For }p \in P, \text{ set }\tau X(p)=\begin{cases}p - \langle u, p \rangle vX + \langle vX, p \rangle u, \text{ if }u\in \mathrm{Um}(P), v\in IP\\p - \langle uX, p \rangle v + \langle v, p \rangle uX,\text{ if }v\in \mathrm{Um}(P), u \in IP.\end{cases}\]
\[\tau X\bigg(\sum_i p_iX^i\bigg)=\sum_i \tau X(p_i)X^i.\] Here, $uX$ represents $u$ times $X$, and $vX$ represents $v$ times $X$. Note that $\tau X$ is a relative orthogonal transvection, since in the first case $u \in$ Um$(P[X])$, $vX \in IP[X]$, and in the second case $v \in$ Um$(P[X])$, $uX \in IP[X]$. We set $\beta(X)$ to be the product of elements of the form $\tau X$, whenever $\tau$ appears in the expression of $\alpha$. Then $\beta(X) \in \TransO(P[X], IP[X], \langle , \rangle)$, with $\beta(1) = \alpha$, and $\beta(0) = Id$.  $\hfill \square$

\begin{lemma}[\cite{Chat}, Theorem 4.23]
\label{rot=reot}
   Let $R$ be a commutative ring with $R = 2R$, and let $I$ be an ideal of
$R$. Let $(P, \langle, \rangle)$ be a locally split orthogonal $R$-module with $P$ of rank $2n$, $n \geq 2$,
and $Q = R^2 \oplus P$ with the induced form on $\mathbb{H}(R) \oplus P$. Then $\TransO(Q, IQ, \langle, \rangle) =
\ETransO(Q, IQ, \langle, \rangle)$.
\end{lemma}

\begin{lemma}[\cite{Chat}, Lemma 4.21]
\label{orth LG}
    Let $R$ be a commutative ring with $R = 2R$, and let $I$ be an ideal of
$R$. Let $(P, \langle, \rangle)$ be a locally split orthogonal $R$-module with $P$ of rank $2n$, $n \geq 2$,
and $Q = R^2 \oplus P$ with the induced form on $\mathbb{H}(R) \oplus P$. Also assume that over the local ring $R_{\mathfrak{m}}$ for any maximal ideal $\mathfrak{m}$ of $R$, the bilinear form $\langle , \rangle$ corresponds to the symmetric matrix $\widetilde{\varphi}_{\mathfrak{m}}$ (with respect to some basis). Let $\alpha(X) \in \operatorname{O}(Q[X])$, with
$\alpha(0) = Id$. If $\alpha(X)_{\mathfrak{m}} \in \operatorname{ETrans}_{\operatorname{O}}(Q_{\mathfrak{m}}[X], IQ_{\mathfrak{m}}[X], \langle, \rangle_{\widetilde{\psi_1} \perp \widetilde{\varphi}_{\mathfrak{m}}})$, for each maximal ideal
$\mathfrak{m}$ of $R$, then $\alpha(X) \in \operatorname{ETrans}_{\operatorname{O}}(Q[X], IQ[X], \langle, \rangle)$.
\end{lemma}

\noindent Now we prove the main result of this section, first by direct calculations. Then, we give a second proof of the result using local-global principle.
\begin{theorem}
\label{ROT normality}
    Let $R$ be a commutative ring with $R = 2R$, and let $I$ be an ideal of
$R$. Let $(P, \langle, \rangle)$ be a locally split orthogonal $R$-module with $P$ of rank $2n$, $n \geq 2$,
and $Q = R^2 \oplus P$ with the induced form on $\mathbb{H}(R) \oplus P$. Then $\operatorname{ETrans}_{\operatorname{O}}(Q, IQ, \langle , \rangle)$ is normal in $\operatorname{O}(Q)$.
\end{theorem}

\noindent \textit{Proof 1:} 
Note that $\operatorname{ETrans}_{\operatorname{O}}(Q, IQ, \langle , \rangle)=\operatorname{Trans}_{\operatorname{O}}(Q, IQ, \langle , \rangle)$  by Lemma \ref{rot=reot}. So, it suffices to show that $\operatorname{Trans}_{\operatorname{O}}(Q, IQ, \langle , \rangle)$ is normal in $\operatorname{O}(Q)$. Let $\gamma\in\operatorname{O}(Q)$, and $\tau\in\rm{Trans}_{\rm{O}}$$(Q,IQ,\langle,\rangle)$ be a relative orthogonal transvection. By definition, we have
\begin{equation*}
	\tau(p)=p-\langle u,p\rangle v+\langle v,p \rangle u
\end{equation*}
where $u, v \in Q$ are isotropic, i.e., $\langle u, u \rangle = \langle v, v \rangle = 0$ with $\langle u, v \rangle = 0$, and either $u$ is unimodular with $v\in IQ$ or $v$ is unimodular with $u\in IQ$.

\vskip 2mm
\noindent Let $p\in Q$, and $q=\gamma^{-1}(p)$. Then,
\begin{eqnarray*}
	\gamma\tau\gamma^{-1}(p)
	&=&\gamma\tau(q)\\
	&=&\gamma (q-\langle u,q\rangle v+\langle v,q \rangle u)\\
	&=&\gamma(q)-\langle u,q\rangle \gamma(v)+\langle v,q \rangle \gamma(u)\hfill \text{ \hspace{0.5cm}[as $\gamma$ is $R$-linear]}\\
	&=&p-\langle \gamma(u),\gamma(q)\rangle \gamma(v)+\langle \gamma(v),\gamma(q) \rangle \gamma(u)\text{ \hspace{0.5cm}[as $\gamma\in\mathrm{O}(Q)$]}\\
    &=&p-\langle \gamma(u),p\rangle \gamma(v)+\langle \gamma(v),p \rangle \gamma(u)
\end{eqnarray*}

\noindent Note that since $u,v$ are isotropic, and $\gamma\in\operatorname{O}(Q)$, we get that $\gamma(u),\gamma(v)$ are isotropic. Also, since  either $u$ is unimodular with $v\in IQ$ or $v$ is unimodular with $u\in IQ$, either $\gamma(u)$ is unimodular with $\gamma(v)\in IQ$, or $\gamma(v)$ is unimodular with $\gamma(u)\in IQ$. Therefore, $\gamma\tau\gamma^{-1}$ is a relative orthogonal transvection, and hence $\operatorname{Trans}_{\operatorname{O}}(Q,IQ,\langle,\rangle)$ is a normal subgroup of $\operatorname{O}(Q)$. $\hfill \square$

\vskip 4mm
\noindent \textit{Proof 2:}
Let $\delta \in \operatorname{ETrans}_{\operatorname{O}}(Q, IQ, \langle , \rangle)$. Note that $\operatorname{ETrans}_{\operatorname{O}}(Q, IQ, \langle , \rangle)=\operatorname{Trans}_{\operatorname{O}}(Q, IQ, \langle , \rangle)$ by Lemma \ref{rot=reot}, and hence $\delta \in \operatorname{Trans}_{\operatorname{O}}(Q, IQ, \langle , \rangle)$. 

\vskip 2mm
\noindent By Lemma \ref{orth homotopy}, there exists $\delta(X) \in \operatorname{Trans}_{\operatorname{O}}(Q[X], IQ[X], \langle , \rangle)$ such that $\delta(1)=\delta$, and $\delta(0)=Id$. For $\gamma \in \operatorname{O}(Q)$, define $\gamma(X)\in\mathrm{O}(Q[X])$ as $\gamma(X)(\sum q_jX^j)=\sum \gamma(q_j)X^j$. Define $\beta(X)=\gamma(X)\delta(X)\gamma(X)^{-1}$. 

\vskip 2mm
\noindent Note that for every maximal ideal $\mathfrak{m}$ of $R$, the projective $R_\mathfrak{m}$-module $Q_\mathfrak{m}$ is free of $\operatorname{rank} \geq 6$. 

Let \(e_1, \ldots, e_{2n+2}\) be the standard basis of \(Q_{\mathfrak{m}}\) with respect to which the bilinear form on \(Q_{\mathfrak{m}}\) corresponds to \(\widetilde{\psi_1} \perp \widetilde{\varphi}\). Since \((P_{\mathfrak{m}}, \langle , \rangle)\) is a split orthogonal \(R\)-module with \(R_{\mathfrak{m}} = 2R_{\mathfrak{m}}\), by Lemma \ref{split} we get \(\widetilde{\varphi} = \varepsilon^\top \widetilde{\psi_n} \varepsilon\), for some \(\varepsilon \in \operatorname{GL}_{2n}(R_{\mathfrak{m}})\). So, by Lemma \ref{reot conj}, we get \[\operatorname{Trans_{O}}(Q_{\mathfrak{m}}[X], IQ_{\mathfrak{m}}[X], \langle , \rangle_{\widetilde{\psi_1} \perp \widetilde{\varphi}}) = (I_2 \perp \varepsilon)^{-1} \operatorname{Trans_{O}}(Q_{\mathfrak{m}}[X], IQ_{\mathfrak{m}}[X], \langle , \rangle_{\widetilde{\psi}_{n+1}}) (I_2 \perp \varepsilon),\]
and by Lemma \ref{orth conj}, we get \[\operatorname{O}_{\widetilde{\psi_1} \perp \widetilde{\varphi}}(R_{\mathfrak{m}}[X]) = (I_2 \perp \varepsilon)^{-1} \operatorname{O}_{\widetilde{\psi}_{n+1} } (R_{\mathfrak{m}}[X]) (I_2 \perp \varepsilon).\]
\vskip 2mm
\noindent Note that by Lemma \ref{KopSus normality}, we have $\operatorname{EO}_{2n+2}(R_{\mathfrak{m}}[X], I_{\mathfrak{m}}[X])$ is normal in $\operatorname{O}_{2n+2}(R_{\mathfrak{m}}[X])$, and by Lemma \ref{reot=eo}, we have \[\operatorname{TransO}(Q_{\mathfrak{m}}[X], IQ_{\mathfrak{m}}[X], \langle , \rangle_{\widetilde{\psi}_{n+1}})=\operatorname{EO}_{2n+2}(R_{\mathfrak{m}}[X], I_{\mathfrak{m}}[X]).\] So, by a similar calculation as done in proof 2 of Theorem \ref{RST normality}, we get that $\beta(X)_{\mathfrak{m}}$ belongs to $\operatorname{TransO}(Q_{\mathfrak{m}}[X], IQ_{\mathfrak{m}}[X], \langle , \rangle_{\widetilde{\psi_{1}}\perp\widetilde{\varphi}})=\operatorname{ETrans}_\mathrm{O}(Q_{\mathfrak{m}}[X], IQ_{\mathfrak{m}}[X], \langle , \rangle_{\widetilde{\psi_{1}}\perp\widetilde{\varphi}})$, for all maximal ideals $\mathfrak{m}$ of $R$. Hence, by Lemma \ref{orth LG}, $\beta(X)$ belongs to $\operatorname{ETrans}_{\operatorname{O}}(Q[X], IQ[X], \langle , \rangle)$, and $\beta(1)=\gamma\delta\gamma^{-1}$ belongs to $\operatorname{ETrans}_{\operatorname{O}}(Q, IQ, \langle , \rangle)$. Therefore, $\operatorname{ETrans}_{\operatorname{O}}(Q, IQ, \langle , \rangle)$ is a normal subgroup of $\operatorname{O}(Q, \langle,\rangle)$.$\hfill \square$

\begin{remark}
In addition to the normal subgroups already considered, one may also consider the DSER orthogonal group, which is normal in \(\mathrm{O}(Q)\). For more details on the DSER orthogonal group, see \cite{AmbRao}.
\end{remark}

\vskip 2mm
\noindent

\begin{center}
\textbf{Acknowledgements}
\end{center}

\noindent The first author gratefully acknowledges the support and infrastructure provided by BITS-Pilani, Hyderabad campus, while also sincerely thanking the Council of Scientific and Industrial Research (CSIR), India, for the financial support provided through the research fellowship (CSIR File No. 09/1026(0035)/2020-EMR-I).

\end{document}